\documentclass[12pt]{article}
\usepackage{graphicx}
\usepackage{amsmath,amsthm,amssymb,amsfonts,euscript,enumerate}
\usepackage{amsthm}
\usepackage{color,wrapfig}
\usepackage{pxfonts}
\usepackage{verbatim}
\newtheorem{thm}{Theorem}[section]
\newtheorem{cor}[thm]{Corollary}
\newtheorem{lem}[thm]{Lemma}

\newtheorem{rmk}[thm]{Remark}

\newcommand{\R}{{\mathbb{R}}}
\newcommand{\Z}{{\mathbb{Z}}}

\newcommand{\2}{\overline}
\newcommand{\3}{\varepsilon}

\def\ni{\noindent}

\begin{document}
\title{Exact decay rate of a nonlinear elliptic equation related to the 
Yamabe flow} 
\author{Shu-Yu Hsu\\
Department of Mathematics\\
National Chung Cheng University\\
168 University Road, Min-Hsiung\\
Chia-Yi 621, Taiwan, R.O.C.\\
e-mail: syhsu@math.ccu.edu.tw}
\date{Jan 11, 2013}
\smallbreak \maketitle
\begin{abstract}
Let $0<m<\frac{n-2}{n}$, $n\ge 3$, $\alpha=\frac{2\beta +\rho}{1-m}$ and 
$\beta>\frac{m\rho}{n-2-mn}$ for some constant $\rho>0$. Suppose $v$ is a radially 
symmetric symmetric 
solution of $\frac{n-1}{m}\Delta v^m+\alpha v+\beta x\cdot\nabla v=0$, $v>0$, in $\R^n$. 
When $m=\frac{n-2}{n+2}$, the metric $g=v^{\frac{4}{n+2}}dx^2$ corresponds to a 
locally conformally flat Yamabe shrinking gradient soliton with positive 
sectional curvature. We prove that the solution $v$ of the above nonlinear 
elliptic equation has the exact decay rate 
$\lim_{r\to\infty}r^2v(r)^{1-m}=\frac{2(n-1)(n(1-m)-2)}{(1-m)(\alpha (1-m)-2\beta)}$. 
\end{abstract}

\vskip 0.2truein



\vskip 0.2truein
\setcounter{section}{0}

\section{Introduction}
\setcounter{equation}{0}
\setcounter{thm}{0}

Recently there is a lot of study on the equation,
\begin{equation}\label{elliptic-eqn}
\frac{n-1}{m}\Delta v^m+\alpha v+\beta x\cdot\nabla v=0,\quad v>0,\quad 
\mbox{ in }\R^n
\end{equation}
where
\begin{equation}\label{m-relation}
0<m<\frac{n-2}{n},\quad n\ge 3,
\end{equation}
and 
\begin{equation}\label{alpha-beta-relation}
\alpha=\frac{2\beta +\rho}{1-m}
\end{equation}
for some constant $\rho\in\R$ by P.~Daskalopoulos and N.~Sesum 
\cite{DS2}, S.Y.~Hsu \cite{H1}, \cite{H2}, M.A.~Peletier and H.~Zhang 
\cite{PZ} and J.L.~Vazquez \cite{V1}. In the paper \cite{DS2} P.~Daskalopoulos 
and N.~Sesum (cf. \cite{CSZ}, \cite{CMM}) proved the important result that 
any locally conformally flat non-compact gradient Yamabe soliton $g$ with 
positive sectional curvature on a $n$-dimensional manifold, $n\ge 3$, must 
be radially symmetric and has the form $g=v^{\frac{4}{n+2}}dx^2$ where $dx^2$ is 
the Euclidean metric on $\R^n$ and $v$ is a radially symmetric solution of 
\eqref{elliptic-eqn} with $m=\frac{n-2}{n+2}$ and $\alpha$, $\beta$, satisfying
\eqref{alpha-beta-relation} for some constant $\rho>0$, $\rho=0$ or $\rho<0$,
depending on whether $g$ is a shrinking, steady, or expanding Yamabe soliton.
 
On the other hand as observed by B.H.~Gilding, L.A.~Peletier and H.~Zhang
\cite{GP}, \cite{PZ}, and others \cite{DS1}, \cite{DS2}, \cite{V1}, \cite{V2}, 
\eqref{elliptic-eqn} also arises in the study of the self-similar solutions of
the degenerate diffusion equation,
\begin{equation}\label{fast-diffusion-eqn}
u_t=\frac{n-1}{m}\Delta u^m\quad\mbox{ in }\R^n\times (0,T).
\end{equation}
For example (cf. \cite{H1}, \cite{V1}) if $v$ is a radially symmetric solution of 
\eqref{elliptic-eqn} with 
\begin{equation*}
\alpha=\frac{2\beta +1}{1-m}>0,
\end{equation*}
then for any $T>0$ the function
\begin{equation}
u(x,t)=(T-t)^{\alpha}v(x(T-t)^{\beta})
\end{equation}
is a solution of \eqref{fast-diffusion-eqn} in $\R^n\times (-\infty,T)$.
We refer the reader to the book \cite{V1} and the paper \cite{H1} for the relation
between solutions of \eqref{elliptic-eqn} and the other self-similar solutions
of \eqref{fast-diffusion-eqn} for the other parameter ranges of $\alpha$,
$\beta$. 

Note that when $v$ is a radially symmetric solution of \eqref{elliptic-eqn}, 
then $v$ satisfies 
\begin{equation}\label{ode}
\frac{n-1}{m}\left((v^m)''+\frac{n-1}{r}(v^m)'\right)
+\alpha v+\beta rv'=0,\quad v>0,\quad\mbox{  in }(0,\infty)
\end{equation}
and 
\begin{equation}\label{initial-cond}
\left\{\begin{aligned}
&v(0)=\eta\\
&v'(0)=0\end{aligned}\right.
\end{equation}
for some constant $\eta>0$. Existence of solutions of \eqref{ode}, 
\eqref{initial-cond}, for the case $n\ge 3$, $0<m\le (n-2)/n$, $\beta>0$
and $\alpha\le\beta (n-2)/m$ is proved by S.Y.~Hsu in \cite{H1}. On
the other hand by the result of \cite{PZ} and Theorem 7.4 of \cite{V1} 
if \eqref{m-relation} holds, then there exists a constant $\2{\beta}$ with 
$\2{\beta}=0$ when $m=\frac{n-2}{n+2}$ 
such that for any $\alpha=\frac{2\beta+1}{1-m}$ and $\beta>\2{\beta}$, 
there exists a unique solution of \eqref{ode}, \eqref{initial-cond}. 
Moreover if $0<\alpha=\frac{2\beta+1}{1-m}$ and $\beta<\2{\beta}$, then 
\eqref{ode}, \eqref{initial-cond}, has no global solution.
 
Since the asymptotic behavior of solutions of \eqref{fast-diffusion-eqn} is 
usually similar to the behavior of the radially symmetric self-similar 
solutions of \eqref{fast-diffusion-eqn}, hence in order to understand 
the asymptotic behavior of solutions of \eqref{fast-diffusion-eqn} and the 
asymptotic behavior of locally conformally flat non-compact gradient Yamabe 
soliton, it is important to study the asymptotic behavior of the solution
of \eqref{ode}, \eqref{initial-cond}.

Exact decay rate of solution of \eqref{ode}, 
\eqref{initial-cond}, for the case 
\begin{equation*}
\alpha=\frac{2\beta}{1-m}>0
\end{equation*}
and the case 
\begin{equation*}
\frac{2\beta}{1-m}>\max (\alpha,0),
\end{equation*} 
with $m$, $n$, satisfying \eqref{m-relation} 
is obtained by S.Y.~Hsu in \cite{H1}. When \eqref{m-relation} and 
\eqref{alpha-beta-relation} hold for some constant $\rho>0$, 
although it is known (\cite{DS2}, \cite{V1}) that solution $v$ of 
\eqref{ode}, \eqref{initial-cond}, satisfies $v(r)=O(r^{-\frac{2}{1-m}})$ as 
$r\to\infty$, nothing is known about the exact decay rate of $v$. In 
\cite{H2} when $m=\frac{n-2}{n+2}$, $\beta>\frac{\rho}{n-2}>0$, S.Y.~Hsu 
by using estimates for the scalar curvature of the metric 
$g=v^{\frac{4}{n+2}}dx^2$ where $v$ is a radially symmetric solution of 
\eqref{elliptic-eqn}, proved that
\begin{equation}\label{decay-rate-soliton}
\lim_{r\to\infty}r^2v(r)=\frac{(n-1)(n-2)}{\rho}.
\end{equation}
In this paper we will extend the above result and prove the exact decay rate 
of radially symmetric solution $v$ of \eqref{elliptic-eqn} when 
\eqref{m-relation} and \eqref{alpha-beta-relation} hold for some 
constant $\rho>0$. More precisely we will prove the following theorem.

\begin{thm}\label{main-thm}
Let $\eta>0$, $\rho>0$, $m$, $n$, $\alpha$, $\beta$, satisfy
\eqref{m-relation}, \eqref{alpha-beta-relation}, and
\begin{equation}\label{beta-lower-bd-cond}
\beta>\frac{m\rho}{n-2-mn}.
\end{equation}
Suppose $v$ is a solution of \eqref{ode}, \eqref{initial-cond}. Then
\begin{equation}\label{decay-rate}
\lim_{r\to\infty}r^2v(r)^{1-m}=\frac{2(n-1)(n(1-m)-2)}{(1-m)(\alpha (1-m)-2\beta)}.
\end{equation}
\end{thm}

\begin{rmk}
The function
\begin{equation}\label{v0-defn}
v_0(x)=\left(\frac{2(n-1)(n(1-m)-2)}{(1-m)(\alpha (1-m)-2\beta)|x|^2}\right)^{\frac{1}{1-m}}
\end{equation}
is a singular solution of \eqref{elliptic-eqn} in $\R^n\setminus\{0\}$. If $v$
is a solution of \eqref{elliptic-eqn}, then for any $\lambda>0$
the function 
\begin{equation}\label{v-lambda-defn}
v_{\lambda}(x)=\lambda^{\frac{2}{1-m}}v(\lambda x)
\end{equation}
is also a solution of \eqref{elliptic-eqn}. 
\end{rmk}

\begin{cor}\label{v-rescale-limit-cor}
Let $\rho$, $m$, $n$, $\alpha$, $\beta$, satisfy \eqref{m-relation}, 
\eqref{alpha-beta-relation}, \eqref{beta-lower-bd-cond}. Suppose $v$ is a radially 
symmetric solution of \eqref{elliptic-eqn}, and $v_0$, $v_{\lambda}$, are given by 
\eqref{v0-defn} and \eqref{v-lambda-defn} respectively.
Then $v_{\lambda}(x)$ converges uniformly on $\R^n\setminus B_R(0)$
to  $v_0(x)$ for any $R>0$ as $\lambda\to\infty$.
\end{cor}

\begin{cor}\label{metric-decay-rate-cor}(cf. \cite{H2})
The metric $g_{ij}=v^{\frac{4}{n+2}}dx^2$, $n\ge 3$, of a locally conformally flat 
non-compact gradient shrinking Yamabe soliton where $v$ is radially symmetric and
satisfies \eqref{elliptic-eqn} with $m=\frac{n-2}{n+2}$, and $\beta>\frac{\rho}{2}>0$, 
$\alpha$, satisfying \eqref{alpha-beta-relation} has the exact decay rate 
\eqref{decay-rate-soliton}.  
\end{cor}

Since the scalar curvature of the metric $g_{ij}=v^{\frac{4}{n+2}}dx^2$, $n\ge 3$, where
$v$ is a radially symmetric solution of \eqref{elliptic-eqn} with $m=\frac{n-2}{n+2}$
is given by (\cite{DS2}, \cite{H2}),
$$
R(r)=(1-m)\left(\alpha +\beta\frac{rv'(r)}{v(r)}\right),
$$
by Corollary \ref{metric-decay-rate-cor} and an argument similar to the proof of 
Lemma 3.4 and Theorem 1.3 of \cite{H2} we obtain the following extensions of Theorem 
1.2 and Theorem 1.3 of \cite{H2}.

\begin{thm}
Let $m=\frac{n-2}{n+2}$, $n\ge 3$, $\beta>\frac{\rho}{2}>0$, $\alpha$, 
satisfy \eqref{alpha-beta-relation}. Let $v$ be a radially symmetric solution of 
\eqref{elliptic-eqn}. Then 
\begin{equation}\label{rv'/v-limit-1}
\lim_{r\to\infty}\frac{rv'(r)}{v(r)}=-\frac{2}{1-m}
\end{equation}
and the scalar curvature $R(r)$ of the metric $g_{ij}=v^{\frac{4}{n+2}}dx^2$ satisfies
$$
\lim_{r\to\infty}R(r)=\rho.
$$
If $K_0$ and $K_1$ are the sectional curvatures of the $2$-planes perpendicular to and 
tangent to the spheres $\{x\}\times S^{n-1}$ respectively, then
\begin{equation*}
\lim_{r\to\infty}K_0(r)=0,
\end{equation*}
and
\begin{equation*}
\lim_{r\to\infty}K_1(r)=\frac{\rho}{(n-1)(n-2)}.
\end{equation*}
\end{thm}

\begin{cor}
Let $\eta>0$, $\rho>0$, $m$, $n$, $\alpha$, $\beta$, satisfy
\eqref{m-relation}, \eqref{alpha-beta-relation}, and
\eqref{beta-lower-bd-cond}.
Suppose $v$ is a solution of \eqref{ode}, \eqref{initial-cond}. Then
\eqref{rv'/v-limit-1} holds.
\end{cor}

The plan of the paper is as follows. We will prove the boundedness of 
the function
\begin{equation}\label{w-defn}
w(r)=r^2v(r)^{1-m}
\end{equation}
where $v$ is the solution of \eqref{elliptic-eqn} in section two. We will also
find the lower bound of $w$ in section two. In section
three we will prove Theorem \ref{main-thm} and Corollary \ref{v-rescale-limit-cor}. 
We will assume that \eqref{m-relation}, \eqref{alpha-beta-relation}, hold 
for some constant $\rho>0$ and let $v$ be a radially symmetric solution of 
\eqref{elliptic-eqn} or equivalently the solution of \eqref{ode}, 
\eqref{initial-cond}, for some $\eta>0$, 
and 
$$
w_{\infty}=\frac{2(n-1)(n(1-m)-2)}{(1-m)(\alpha (1-m)-2\beta)}
$$
for the rest of the paper. Note that when $\alpha=n\beta$ and 
$\alpha=\frac{2\beta+1}{1-m}$, the solution of 
\eqref{elliptic-eqn} is given explicitly by (cf. \cite{DS2})
$$
v_{\lambda}(x)=\left(\frac{2(n-1)(n-2-nm)}{(1-m)(\lambda^2+|x|^2)}\right)^{\frac{1}{1-m}},
\quad\lambda>0,
$$
which satisfies \eqref{decay-rate}.

\section{$L^{\infty}$ estimate of $w$}
\setcounter{equation}{0}
\setcounter{thm}{0}

\begin{lem}\label{sequence-limit-thm}
Let $\rho>0$, $m$, $n$, $\alpha$, $\beta$, satisfy \eqref{m-relation}
and \eqref{alpha-beta-relation} and let $v$ be a radially symmetric solution of 
\eqref{elliptic-eqn}. Let $w$ be given by \eqref{w-defn}. 
Suppose there exists a constant $C_1>0$ such that
\begin{equation}\label{w-upper-bd}
w(r)\le C_1\quad\forall r\ge 1.
\end{equation}
Then any sequence $\{w(r_i)\}_{i=1}^{\infty}$, $r_i\to\infty$ as 
$i\to\infty$, has a subsequence $\{w(r_i')\}_{i=1}^{\infty}$ such that
\begin{equation}\label{sequence-limit}
\lim_{r\to\infty}w(r_i')=\left\{\begin{aligned}
&0\quad\mbox{ or }\quad w_{\infty}\quad\mbox{ if }\,\,v\not\in L^1(\R^n)\\
&0\quad\mbox{ or }\quad w_1\quad\mbox{ if }\,\,v\in L^1(\R^n)
\quad\mbox{ and }\,\,\beta>0\\
&0\qquad\qquad\quad\mbox{ if }\,\,v\in L^1(\R^n)\quad\mbox{ and }\,\,\beta\le 0\\
\end{aligned}\right.
\end{equation}
where
\begin{equation}\label{w1-defn}
w_1=\frac{2(n-1)}{(1-m)\beta}\quad\mbox{ if }\,\,\beta>0.
\end{equation}
\end{lem}
\begin{proof}
Let $\{r_i\}_{i=1}^{\infty}$ be a sequence such that $r_i\to\infty$ as 
$i\to\infty$. By \eqref{w-upper-bd} the sequence $\{w(r_i)\}_{i=1}^{\infty}$
has a subsequence which we may assume without loss of generality to be the 
sequence itself that converges to some constant $a\in [0,C_1]$ as $i\to\infty$.
Integrating \eqref{ode} over $(0,r)$ and simplifying,
\begin{equation}\label{v'-eqn}
-\frac{n-1}{m}(v^m)'(r)=\beta rv(r)
+\frac{\alpha-n\beta}{r^{n-1}}\int_0^rz^{n-1}v(z)\,dz\quad\forall r>0.
\end{equation}
Integrating \eqref{v'-eqn} over $(r,\infty)$, by \eqref{w-upper-bd}
we get
\begin{equation}\label{v-integral-eqn}
\frac{n-1}{m}v(r)^m=\beta\int_r^{\infty}s v(s)\,ds
+\int_r^{\infty}\frac{\alpha-n\beta}{s^{n-1}}
\left(\int_0^sz^{n-1}v(z)\,dz\right)\,d s\quad\forall r>0.
\end{equation}
Let $b=a^{\frac{1}{1-m}}=\lim_{i\to\infty}r_i^{\frac{2}{1-m}}v(r_i)$. Then by 
\eqref{w-upper-bd}, \eqref{v-integral-eqn}, and the l'Hospital rule,

\begin{align}\label{limit-eqn1}
\frac{(n-1)}{m}b^m=&\frac{(n-1)}{m}\lim_{i\to\infty}(r_i^{\frac{2}{1-m}}v(r))^m\notag\\
=&\beta\lim_{i\to\infty}\frac{\int_{r_i}^{\infty}sv(s)\,ds}{r_i^{-\frac{2m}{1-m}}}
+\lim_{i\to\infty}\frac{\int_{r_i}^{\infty}\frac{\alpha-n\beta}{s^{n-1}}
\left(\int_0^sz^{n-1}v(z)\,dz\right)\,ds}{r_i^{-\frac{2m}{1-m}}}\notag\\
=&\frac{(1-m)}{2m}\left(\beta\lim_{i\to\infty}\frac{r_iv(r_i)}{r_i^{-\frac{2m}{1-m}-1}}
+(\alpha-n\beta)\lim_{i\to\infty}\frac{\frac{1}{r_i^{n-1}}\int_0^{r_i}z^{n-1}v(z)\,dz}
{r_i^{-\frac{2m}{1-m}-1}}\right)\notag\\
=&\frac{(1-m)}{2m}\left(\beta b
+(\alpha-n\beta)\lim_{i\to\infty}\frac{\int_0^{r_i}z^{n-1}v(z)\,dz}
{r_i^{n-\frac{2}{1-m}}}\right)
\end{align}
We now divide the proof into two cases.

\noindent{\bf Case 1}: $v\not\in L^1(\R^n)$.

\noindent By \eqref{limit-eqn1} and the l'Hospital rule,

\begin{align}\label{limit1}
\frac{(n-1)}{m}b^m=&\frac{(1-m)}{2m}\left(\beta b
+\frac{\alpha-n\beta}{n-\frac{2}{1-m}}\cdot
\lim_{i\to\infty}\frac{r_i^{n-1}v(r_i)}{r_i^{n-\frac{2}{1-m}-1}}\right)\notag\\
=&\frac{(1-m)}{2m}\left(\beta b
+\frac{\alpha-n\beta}{n-\frac{2}{1-m}}b\right)\notag\\
=&\frac{(1-m)[\alpha (1-m)-2\beta]}{2m[n(1-m)-2]}b\notag\\
\Rightarrow\qquad a=b=0\quad\mbox{ or }\quad &a=b^{1-m}
=w_{\infty}.
\end{align}

\noindent{\bf Case2}: $v\in L^1(\R^n)$.

\noindent By \eqref{limit-eqn1},

\begin{equation}\label{limit2}
\frac{(n-1)}{m}b^m=\frac{(1-m)\beta}{2m}b\quad
\Rightarrow\quad\left\{\begin{aligned} 
&a=b=0\quad\mbox{ or }\quad a=b^{1-m}=w_1\quad\mbox{ if }\beta>0\\
&a=b=0\qquad\qquad\qquad\qquad\qquad\mbox{ if }\beta\le 0
\end{aligned}\right.
\end{equation}
By \eqref{limit1} and \eqref{limit2} the lemma follows.

\end{proof}

\begin{rmk}\label{w1-w-infty-compare-rmk}
When $\beta>0$, $w_1>w_{\infty}$ if and only if $\alpha>n\beta$.
\end{rmk}

\begin{cor}\label{decay-rate-cor}
Suppose there exist constants $C_1>C_2>0$ such that
$$
C_2\le w(r)\le C_1\quad\forall r\ge 1.
$$
Then \eqref{decay-rate} holds.
\end{cor}

\begin{lem}\label{w-lower-bd-lem}
Let $\eta>0$, $\rho>0$, $\beta>0$, $m$, $n$, $\alpha\le n\beta$, 
satisfy \eqref{m-relation} and \eqref{alpha-beta-relation}. Then 
\begin{equation}\label{v-lower-bd}
v(r)\ge\left(\eta^{m-1}+\frac{(1-m)\beta}{2(n-1)}r^2\right)^{-\frac{1}{1-m}}
\quad\forall r\ge 0.
\end{equation}
Hence there exists a constant $C_2>0$ such that
\begin{equation}\label{w-lower-bd}
w(r)\ge C_2\quad\forall r\ge 1.
\end{equation}
\end{lem}
\begin{proof}
\eqref{v-lower-bd} is proved on P.22 of \cite{DS2}. For the sake of
completeness we will give a simple different proof here.
By \eqref{v'-eqn},
\begin{align*}
-\frac{n-1}{m}(v^m)'(r)\le&\beta rv(r)\quad\forall r>0\\
\Rightarrow\qquad\quad -(n-1)v^{m-2}v'(r)\le&\beta r\quad\forall r>0\\
\Rightarrow\quad\frac{n-1}{1-m}(v(r)^{m-1}-\eta^{m-1})\le&\frac{\beta}{2}r^2
\quad\forall r>0
\end{align*}
and \eqref{v-lower-bd} follows. By \eqref{v-lower-bd}, we get 
\eqref{w-lower-bd} and the lemma follows.
\end{proof}

We now recall a result of \cite{H2}.

\begin{lem}(cf. Lemma 2.3 of \cite{H2})\label{w-upper-bd-lem}
Let $\eta>0$, $\rho>0$, $m$, $n$, $\alpha\ge n\beta>0$, satisfy 
\eqref{m-relation} and \eqref{alpha-beta-relation}.
Then there exists a constant $C_1>0$ such that \eqref{w-upper-bd} holds. 
\end{lem}
\begin{proof}
This result is proved in \cite{H2}. For the sake of completeness we will repeat the 
proof here. By \eqref{v'-eqn}, $v'(r)<0$ for all $r>0$. Then by 
\eqref{v'-eqn},
\begin{align*}
&\frac{n-1}{m}r^{n-1}(v^m)'(r)\le-\beta r^nv(r)-(\alpha-n\beta)
\int_0^rz^{n-1}v(r)\,dz=-\frac{\alpha}{n}r^nv(r)\quad\forall r>0\\
\Rightarrow\quad&v^{m-2}(r)v'(r)\le-\frac{\alpha}{n(n-1)}r
\qquad\qquad\quad\forall r>0\\
\Rightarrow\quad&v(r)\le\left(\eta^{m-1}+\frac{\alpha(1-m)}{2n(n-1)}r^2\right)^{-\frac{1}{1-m}}
\le\left(\frac{2n(n-1)}{\alpha(1-m)}r^{-2}\right)^{\frac{1}{1-m}}\quad\forall r>0.
\end{align*}
Hence \eqref{w-upper-bd} holds with $C_1=\frac{2n(n-1)}{\alpha(1-m)}$ and the lemma follows.
\end{proof}

\begin{lem}\label{w-upper-bd-lem0}
Let $\eta>0$, $\rho>0$, $m$, $n$, $0<\alpha\le n\beta$, satisfy 
\eqref{m-relation} and \eqref{alpha-beta-relation}.
Then there exists a constant $C_1>0$ such that \eqref{w-upper-bd} holds.
\end{lem}
\begin{proof}
Let $A=\{r\in [1,\infty):w'(r)\ge 0\}$. We now divide the proof into two cases.

\noindent{\bf Case 1}: $A\cap [R_0,\infty)\ne\phi\quad\forall R_0>1$.

\noindent We will use a modification of the proof of Lemma 3.2 of \cite{H2}
to proof this case. By Lemma \ref{w-lower-bd-lem} there exists a constant $C_2>0$
such that \eqref{w-lower-bd} holds. Hence by \eqref{w-lower-bd}, 
\begin{align}\label{r^n-v-infty}
&r^nv(r)=r^{n-\frac{2}{1-m}}w(r)^{\frac{1}{1-m}}\ge C_2r^{n-\frac{2}{1-m}}\quad\forall 
r\ge 1\notag\\
\Rightarrow\quad&r^nv(r)\to\infty\quad\mbox{ as }r\to\infty.
\end{align}
We now claim that 
\begin{equation}\label{term-ratio}
\limsup_{\substack{r\in A\\r\to\infty}}
\frac{\int_0^rz^{n-1}v(z)\,dz}{r^nv(r)}\le\frac{1-m}{n(1-m)-2}.
\end{equation}
We divide the proof of the above claim into two cases.

\noindent{\bf Case (1a)}: $\int_0^{\infty}z^{n-1}v(z)\,dz<\infty$.

\noindent By \eqref{r^n-v-infty} we get \eqref{term-ratio}.

\noindent{\bf Case (1b)}: $\int_0^{\infty}z^{n-1}v(z)\,dz=\infty$.

\noindent Since 
$$
\frac{d}{dr}(r^nv(r))=\left(n-\frac{2}{1-m}\right)r^{n-1}v(r)
+\frac{1}{1-m}r^{n-\frac{2}{1-m}}w^{\frac{m}{1-m}}(r)w'(r)\ge 
\left(n-\frac{2}{1-m}\right)r^{n-1}v(r)\quad\forall r\in A,
$$
by \eqref{r^n-v-infty} and the l'Hospital rule,
\begin{align*}
\limsup_{\substack{r\in A\\r\to\infty}}\frac{\int_0^rz^{n-1}v(z)\,dz}{r^nv(r)}
=&\limsup_{\substack{r\in A\\r\to\infty}}\frac{r^{n-1}v(r)}{\left(n-\frac{2}{1-m}\right)
r^{n-1}v(r)+\frac{1}{1-m}r^{n-\frac{2}{1-m}}w^{\frac{m}{1-m}}(r)w'(r)}\\
\le&\left(n-\frac{2}{1-m}\right)^{-1}
\end{align*}
and \eqref{term-ratio} follows. Let $0<\delta<\frac{\rho}{n(1-m)-2}$. 
By \eqref{term-ratio} there exists a constant $R_1>1$ such that
\begin{align}\label{term-compare}
&\frac{\int_0^rz^{n-1}v(z)\,dz}{r^nv(r)}<\frac{(1-m)}{n(1-m)-2}
+\frac{\delta}{1+n\beta-\alpha}
\qquad\qquad\quad\forall r\ge R_1, r\in A\notag\\
\Rightarrow\quad&\int_0^rz^{n-1}v(z)\,dz\le\left(
\frac{(1-m)}{n(1-m)-2}+\frac{\delta}{1+n\beta-\alpha}\right)r^nv(r)
\quad\forall r\ge R_1, r\in A.
\end{align}
By \eqref{v'-eqn} and \eqref{term-compare},
\begin{align*}
\frac{n-1}{m}r^{n-1}(v^m)'(r)
\le&-\beta r^nv(r)+\left(\frac{(n\beta -\alpha)(1-m)}{n(1-m)-2}+\delta\right)
r^nv(r)\notag\\
\le&-\left(\frac{\rho}{n(1-m)-2}-\delta\right)r^nv(r)
\quad\forall r\ge R_1, r\in A\notag\\
\Rightarrow\quad (n-1)v^{m-2}v'(r)
\le&-\left(\frac{\rho}{n(1-m)-2}-\delta\right)r\qquad\quad\forall r\ge R_1, r\in A.
\end{align*}
Hence there exists a constant $C_3>0$ such that
\begin{align}
&\frac{rv'(r)}{v(r)}\le -C_3r^2v(r)^{1-m}=-C_3w(r)\quad\forall r\ge R_1,r\in A\notag\\
\Rightarrow\quad&0\le w'(r)=\frac{2w(r)}{r}\left(1+\frac{1-m}{2}\cdot
\frac{rv'(r)}{v(r)}\right)\le\frac{2w(r)}{r}\left(1-\frac{(1-m)C_3}{2}w(r)\right)
\quad\forall r\ge R_1,r\in A\notag\\
\Rightarrow\quad&w(r)\le\frac{2}{(1-m)C_3}\quad\forall r\ge R_1,r\in A.
\label{w-upper-bd1}
\end{align}
Let $r_1\in A\cap [R_1,\infty)$. Then for any $r'\in (r_1,\infty)\setminus A$,
there exists $r_2\in A\cap [r_1,\infty)$ such that 
\begin{align}\label{w-upper-bd2}
&w'(r)<0\quad\forall r_2<r\le r'\quad\mbox{ and }\quad w'(r_2)=0\notag\\
\Rightarrow\quad&w(r')\le w(r_2)\le\frac{2}{(1-m)C_3}\quad\forall r'>r_1,r'\not\in A
\quad\mbox { (by \eqref{w-upper-bd1})}.
\end{align}
By \eqref{w-upper-bd1} and \eqref{w-upper-bd2},
$$
w(r)\le\frac{2}{(1-m)C_3}\quad\forall r\ge r_1
$$
and \eqref{w-upper-bd} holds with 
$C_1=\max\left(\frac{2}{(1-m)C_3},\max_{1\le r\le r_1}w(r)\right)$. 

\noindent{\bf Case 2}: There exists a constant $R_0>1$ such that $A\cap [R_0,\infty)
=\phi$.

\noindent Then $w'(r)<0$ for all $r\ge R_0$. Hence \eqref{w-upper-bd} holds with 
$C_1=\max_{1\le r\le R_0}w(r)$ and the lemma follows.
\end{proof}

\section{Proof of Theorem \ref{main-thm}}
\setcounter{equation}{0}
\setcounter{thm}{0}

We first recall a result of \cite{H1}:

\begin{lem}\label{h1-lower-bd}(cf. Lemma 2.1 of \cite{H1})
Let $\eta>0$, $m$, $n$, $\alpha>0$, $\beta\ne 0$ satisfy \eqref{m-relation} and 
\begin{equation*}\label{alpha-beta-relation4}
\frac{m\alpha}{\beta}\le n-2.
\end{equation*}
Let $v$ be the solution of \eqref{ode}, \eqref{initial-cond}. Then  
\begin{equation}\label{basic-monotone-ineqn}
v(r)+\frac{\beta}{\alpha}rv'(r)>0\quad\forall r\ge 0
\end{equation}
and
\begin{equation}\label{v'-negative}
v'(r)<0\quad\quad\forall r>0.
\end{equation}
\end{lem} 

\begin{lem}\label{r^{n-2}-vm-infty-limit-thm}
Let $\rho>0$, $m$, $n$, $\alpha>n\beta$, satisfy \eqref{m-relation}, 
\eqref{alpha-beta-relation} and \eqref{beta-lower-bd-cond}.
Then
\begin{equation}\label{r^{n-2}-vm-infty-limit}
\lim_{r\to\infty}r^{n-2}v^m(r)=\infty.
\end{equation} 
\end{lem}
\begin{proof}
Suppose \eqref{r^{n-2}-vm-infty-limit} does not hold. Then there exists a sequence 
$\{r_i\}_{i=1}^{\infty}$, $r_i\to\infty$ as $i\to\infty$, such that
$r_i^{n-2}v^m(r_i)\to a_1$ as $i\to\infty$ for some constant $a_1\ge 0$. By
Lemma \ref{sequence-limit-thm} the sequence $\{r_i\}_{i=1}^{\infty}$ has a subsequence 
which we may assume without loss of generality to be the sequence itself such that
$w(r_i)\to a_2$ as $i\to\infty$ where $a_2=0$, $w_{\infty}$, or $w_1$ with $w_1$ being given
by \eqref{w1-defn}. By \eqref{v-integral-eqn}, Lemma \ref{w-upper-bd-lem}, 
Lemma \ref{w-upper-bd-lem0}, and the l'Hospital rule, 
\begin{align*}
\frac{(n-1)}{m}a_1
=&\frac{(n-1)}{m}\lim_{i\to\infty}r_i^{n-2}v(r_i)^m\notag\\
=&\beta\lim_{i\to\infty}\frac{\int_{r_i}^{\infty}s v(s)\,ds}{r_i^{2-n}}
+\lim_{i\to\infty}\frac{\int_{r_i}^{\infty}\frac{\alpha-n\beta}{s^{n-1}}
\left(\int_0^sz^{n-1}v(z)\,dz\right)\,ds}{r_i^{2-n}}\notag\\
=&\frac{\beta}{n-2}\lim_{i\to\infty}r_i^nv(r_i)
+\frac{\alpha-n\beta}{n-2}\lim_{i\to\infty}\int_0^{r_i}z^{n-1}v(z)\,dz\notag\\
=&\frac{\beta}{n-2}\lim_{i\to\infty}r_i^{n-2}v(r_i)^m\cdot\lim_{i\to\infty}r_i^2v(r_i)^{1-m}
+\frac{\alpha-n\beta}{n-2}\int_0^{\infty}z^{n-1}v(z)\,dz\notag\\
=&\frac{\beta}{n-2}a_1a_2+\frac{\alpha-n\beta}{n-2}\int_0^{\infty}z^{n-1}v(z)\,dz.
\end{align*}
Hence 
\begin{equation}\label{ratio-limit}
\frac{\alpha-n\beta}{a_1}\int_0^{\infty}z^{n-1}v(z)\,dz
=\frac{(n-1)(n-2)}{m}-\beta a_2.
\end{equation}
By \eqref{v'-eqn} and \eqref{ratio-limit},
\begin{align}
-(n-1)\lim_{i\to\infty}\frac{r_iv'(r_i)}{v(r_i)}=&\beta\lim_{i\to\infty}r_i^2v(r_i)^{1-m}
+\lim_{i\to\infty}\frac{(\alpha-n\beta)}{r_i^{n-2}v(r_i)^m}\int_0^{r_i}z^{n-1}v(z)\,dz
\label{rv'/v-identity}\\
=&\frac{(n-1)(n-2)}{m}.\notag
\end{align}
Hence
\begin{equation}\label{rv'/v-eqn1}
\lim_{i\to\infty}\frac{r_iv'(r_i)}{v(r_i)}=-\frac{(n-2)}{m}.
\end{equation}
By \eqref{m-relation}, \eqref{alpha-beta-relation} and \eqref{beta-lower-bd-cond}, 
\begin{equation*}\label{alpha-beta-relation5}
\frac{m\alpha}{\beta}<n-2
\end{equation*}
holds. Hence there exists a constant $\3>0$ such that 
\begin{equation}\label{alpha-beta-relation6}
\frac{m\alpha}{\beta}<n-2-\3.
\end{equation}
By \eqref{alpha-beta-relation6} and Lemma \ref{h1-lower-bd}, \eqref{basic-monotone-ineqn}
and \eqref{v'-negative} hold. Then by \eqref{basic-monotone-ineqn}, \eqref{v'-negative} 
and \eqref{alpha-beta-relation6},
\begin{align}
&0>\frac{rv'(r)}{v(r)}>-\frac{\alpha}{\beta}>-\frac{n-2}{m}+\frac{\3}{m}\quad\forall
r>0\label{rv'/v-upper-lower-bd}\\
\Rightarrow\quad&\lim_{i\to\infty}\frac{r_iv'(r_i)}{v(r_i)}\ge-\frac{n-2}{m}+\frac{\3}{m}.
\notag
\end{align}
which contradicts \eqref{rv'/v-eqn1}. Hence no such sequence $\{r_i\}_{i=1}^{\infty}$ 
exists and the lemma follows.
\end{proof}

\begin{lem}\label{w>epsilon-infinite-number-times-lem}
Let $\rho>0$, $m$, $n$, $\alpha>n\beta$, satisfy \eqref{m-relation}, 
\eqref{alpha-beta-relation} and \eqref{beta-lower-bd-cond}.
Then there exists a constant $\3\in (0,\min (1,w_{\infty}/2))$ such that for any 
$R_0>1$ there exists $r'>R_0$ such that
\begin{equation*}
w(r')\ge\3.
\end{equation*}
\end{lem}
\begin{proof}
Suppose the lemma is false. Then
\begin{equation}\label{w-limit-0}
\lim_{r\to\infty}w(r)=0.
\end{equation}
We claim that
\begin{equation}\label{rv'/v-eqn2}
\lim_{r\to\infty}\frac{rv'(r)}{v(r)}=0.
\end{equation} 
By the proof of Lemma \ref{r^{n-2}-vm-infty-limit-thm} there exists a constant $\3>0$
such that \eqref{rv'/v-upper-lower-bd} holds. Suppose \eqref{rv'/v-eqn2} does not
hold. Then by \eqref{rv'/v-upper-lower-bd} and \eqref{w-limit-0} there exists a sequence 
$\{r_i\}_{i=1}^{\infty}$, $r_i\to\infty$ as $i\to\infty$, 
such that $r_iv'(r_i)/v(r_i)\to a_3$ as $i\to\infty$ for some constant $a_3$ satisfying
\begin{equation}\label{a3-lower-upper-bd}
-\frac{n-2}{m}+\frac{\3}{m}\le a_3<0
\end{equation}
and \eqref{rv'/v-identity} holds. By Lemma \ref{r^{n-2}-vm-infty-limit-thm}, 
\eqref{rv'/v-identity}, \eqref{w-limit-0} and \eqref{a3-lower-upper-bd}, we get
\begin{equation*}
-(n-1)\lim_{i\to\infty}\frac{r_iv'(r_i)}{v(r_i)}=0\qquad\qquad\mbox{ if } v\in L^1(\R^n)
\end{equation*}
and  if $v\not\in L^1(\R^n)$, then by the l'Hopsital rule, 
\begin{align*}
-(n-1)\lim_{i\to\infty}\frac{r_iv'(r_i)}{v(r_i)}
=&(\alpha-n\beta)\lim_{i\to\infty}\frac{r_i^{n-1}v(r_i)}{(n-2)r_i^{n-3}v(r_i)^m+mr_i^{n-2}
v(r_i)^{m-1}v'(r_i)}\\
=&(\alpha-n\beta)\lim_{i\to\infty}\frac{r_i^2v(r_i)^{1-m}}{n-2+m(r_iv'(r_i)/v(r_i))}\\
=&\frac{\alpha-n\beta}{n-2+ma_3}\cdot\lim_{i\to\infty}r_i^2v(r_i)^{1-m}\\
=&0.
\end{align*}
Hence
\begin{equation*}
a_3=\lim_{i\to\infty}\frac{r_iv'(r_i)}{v(r_i)}=0
\end{equation*}
which contradicts \eqref{a3-lower-upper-bd}. Thus no such sequence $\{r_i\}_{i=1}^{\infty}$ 
exists and \eqref{rv'/v-eqn2} follows. Since
\begin{equation*}
w'(r)=\frac{2w(r)}{r}\left(1+\frac{1-m}{2}\cdot\frac{rv'(r)}{v(r)}\right),
\end{equation*}
by \eqref{rv'/v-eqn2} there exists a constant $R_0>0$ such that such that 
\begin{equation*}
w'(r)>0\quad\forall r\ge R_0
\end{equation*}
which contradicts \eqref{w-limit-0} and the lemma follows.
\end{proof}

We are now ready for the proof of Theorem \ref{main-thm}.

\noindent {\ni{\it Proof of Theorem \ref{main-thm}}:}
We divide the proof into two cases.

\noindent{\bf Case 1}: $\alpha\le n\beta$.

By Corollary \ref{decay-rate-cor} , Lemma \ref{w-lower-bd-lem} and 
Lemma \ref{w-upper-bd-lem0}, we get \eqref{decay-rate}.

\noindent{\bf Case 2}: $\alpha>n\beta$.

By  Lemma \ref{w-upper-bd-lem} there exists a constant 
$C_1>0$ such that \eqref{w-upper-bd} holds. Let $0<\3<\min (1,w_{\infty}/2)$
be as in Lemma \ref{w>epsilon-infinite-number-times-lem}.
Suppose there exists a sequence $\{r_i\}_{i=1}^{\infty}$, $r_i\to\infty$ as
$i\to\infty$, such that $w(r_i)<\3$ for all $i\in\Z^+$. 
Then by Lemma \ref{w>epsilon-infinite-number-times-lem} there exists a subsequence
of $\{r_i\}_{i=1}^{\infty}$ which we may assume without loss of generality to be the
sequence itself and a sequence $\{r_i'\}_{i=1}^{\infty}$ such that $r_i<r_i'<r_{i+1}$
for all $i=1,2,\dots$ and 
\begin{equation}\label{w<epsilon<w}
w(r_i)<\3<w(r_i')\quad\forall i=1,2,\dots.
\end{equation}
By \eqref{w<epsilon<w} and the intermediate value theorem, for any $i=1,2,\dots$,
there exists $a_i\in (r_i,r_i')$ such that
\begin{equation*}
w(a_i)=\3\quad\forall i=1,2,\dots.
\end{equation*}
Hence $a_i\to\infty$ as $i\to\infty$ and
\begin{equation*}
\lim_{i\to\infty}w(a_i)=\3.
\end{equation*}
This contradicts Lemma \ref{sequence-limit-thm} and 
Remark \ref{w1-w-infty-compare-rmk}. Hence no such sequence $\{r_i\}_{i=1}^{\infty}$ 
exists. Thus there exists a constant $R_1>1$ such that 
$w(r)\ge\3$ for all $r\ge R_1$. Hence \eqref{w-lower-bd} holds with 
$C_2=\min (\3,\min_{1\le r\le R_1}w(r))>0$. By Corollary \ref{decay-rate-cor} we get 
\eqref{decay-rate} and the theorem follows.

{\hfill$\square$\vspace{6pt}}

\noindent {\ni{\it Proof of Corollary \ref{v-rescale-limit-cor}}:}
By Theorem \ref{main-thm}, 
$$
|x|^2v_{\lambda}(x)^{1-m}=(\lambda |x|)^2v(\lambda x)^{1-m}\to 
\frac{2(n-1)(n(1-m)-2)}{(1-m)(\alpha (1-m)-2\beta)}
\quad\mbox{ uniformly on }\R^n\setminus B_R(0) 
$$
as $\lambda\to\infty$ for any $R>0$ and the corollary follows.

{\hfill$\square$\vspace{6pt}}

\end{document}